\documentclass[oneside,10pt,final]{amsart}
\makeatletter
\@namedef{subjclassname@2020}{%
  \textup{2020} Mathematics Subject Classification}
\makeatother
\usepackage{tikz-cd}
\usepackage[notcite,notref]{showkeys}
\usepackage[a4paper,width=160mm,top=27mm,bottom=27mm]{geometry}
\usepackage{cases}
\usepackage{amsmath,amssymb}
\usepackage{comment}
\usepackage[hidelinks]{hyperref}
\newtheorem{theorem}{Theorem}[section]
\newtheorem{lemma}[theorem]{Lemma}

\newtheorem{corollary}[theorem]{Corollary}

\theoremstyle{definition}
\newtheorem{remx}[theorem]{Remark}
\newenvironment{remark}
  {\pushQED{\qed}\remx}
  {\popQED\endremx}

\newcommand{\N}{\mathbb{N}}
\newcommand{\R}{\mathbb{R}}

\numberwithin{equation}{section}

\begin{document}

\address{Mouhamadou Sy
\newline \indent AIMS-Rwanda
\newline \indent 17 KN 16 Ave, Nyarugenge District |Kigali, Rwanda
\newline\indent P.O Box,7150, Kigali, Rwanda}
\email{mouhamadous314@gmail.com, mouhamadou.sy@aims.ac.rw}

\title[Regularity and Pathwise bounds for probabilistic solutions]{Regularity and Pathwise bounds for probabilistic solutions of  PDEs}
\author{Mouhamadou Sy}

\begin{abstract}
In this paper, we build a procedure that allows to establish regularity and controls in time for probabilistic solutions to PDEs. Probabilistic approaches to global wellposedness problems usually provide ensemble bounds on the solutions. These bounds are the main tools to ensure convergence procedures yielding the existence and uniqueness of global solutions. A question of interest consists in transforming such ensemble bounds into individual controls on the flow ; this, among other uses, gives valuable information on the long-time behavior of the solutions. Toward such question of bounds transformation, Bourgain initiated a successful procedure that exploited the local wellposedness of the PDE, with an estimate of the time of size-doubling. In this note, we construct an estimation procedure which relies on a different local requirement. It turns out that this substitute is flexible enough to be possible to fulfill with the help of the ensemble bound itself. For applications of the procedure, we are able to provide new pathwise controls on solutions to NLS equations.
\end{abstract}
\bigskip


\maketitle

%\setcounter{tocdepth}{1}
%\tableofcontents

\section{Introduction}
In the analysis of differential equations, a key tool consists in changing one form of an inequality into another more suitable one. Such kind of trade gave outstanding usefulness to the famous Gronwall inequality and related estimates.

In a context where the differential equation is subject to randomness, the time variable is accompanied by an ensemble structure that, in principle, generates ensemble bounds on the solutions. The main question of this work concerns the change of these ensemble bounds into individual (pathwise) ones : bounds on the time evolution of the realizations.

Individual bounds on solutions can be instrumental in proving fine properties of the evolution, ranging from globalization arguments to long-time behavior properties. For instance, the so-called weak turbulence hypothesis in dispersive PDEs posed on bounded domains can be analyzed through the lens of the growth of Sobolev norms. For problem of the growth of norms, exponential bounds can be obtained from an iteration of a uniform local increment property. More interesting are the  polynomial bounds obtained via bilinear estimations \cite{staffilani1997quadratic}, the high-low method \cite{bourgain1998refinements} or the $I-$method \cite{colliander2002almost}. Closer to our subject of study, logarithmic bounds have been obtained in probabilistic settings, precisely through  an individualization of ensemble bounds.

In the Gibbs measures theory for PDEs, the need of such change of the nature of bounds has been a barrier to establish global existence of solutions. Bourgain \cite{Bourgain1994} introduced a clever exploitation of local increment property of the flow to successfully transform Gaussian type ensemble bounds into logarithmic controls on trajectories (see also the construction of \cite{tzvetkov2008invariant}).

Several works have been able to establish unique global solutions to PDEs via probabilistic strategies which do not rely on local wellposedness; however, in the best of our knoweledge, none of these established individual controls on the time-evolution of the solutions (see e.g. the well-known fluctuation-dissipation method \cite{Kuksin2004Eulerian,KuksinShirikyan2004,KuksinShirikyan2012,FoldesSy2021,gussetti2025statistical}).

The aim of the present work is to build a procedure employing a different local condition, making it flexible enough to apply to various situations not necessarily built upon the framework of \cite{Bourgain1994}. Here are two main features of the procedure:
\begin{itemize}
    \item Any ensemble control on an increasing function of the norm is admissible for obtaining individual bounds ; 
    \item  The local increment requirement can in practice be established by exploiting the ensemble bound itself. This implies important flexibility, as probabilistic estimates, which are the decisive ingredients of the procedure, find a new utilization.

\end{itemize}

As for applications of the estimation procedure, we will consider in Section \ref{SecApp} the following two situations :
\begin{enumerate}
    \item {\bfseries Context of an inaccessible local theory} - cubic NLS after Kuksin-Shirikyan \cite{KuksinShirikyan2004}. Using probabilistic and compactness arguments, the authors of \cite{KuksinShirikyan2004} constructed unique global strong solutions to 
    \begin{align}
        \partial_t u=i(\Delta-|u|^2)u
    \end{align}
    satisfying an ensemble bound
    \begin{align}
        \mathbb{E}\|u\|_{H^2}^2<\infty.
    \end{align}
    Our task will be to change this bound with a control on the trajectories.
    \item {\bfseries Context of a weak ensemble bound} \cite{sy2021almost}. It was build a probabilistic global wellposedness for 
     \begin{align}
        \partial_t u=i(\Delta-|u|^{p-1})u,\quad p>1,
    \end{align}
    on $H^s(\mathbb{T}^3)$, $s>\frac{3}{2}$.
    The solutions satisfy the ensemble bounds :
    \begin{align}
        \mathbb{E}\|u\|_{H^s}^2 &<\infty\label{Est_quad} \\
        \mathbb{E}e^{\xi^{-1}(\|u\|_{H^{s-}}^2)} &<\infty\label{Est_exp}
    \end{align}
    where $\xi^{-1}:\R_+\to\R_+$ is a convex function. In \cite{sy2021almost}, the estimate \eqref{Est_exp} has been exploited within the Bourgain procedure to produce individual bounds. However, due to the relatively weak control of \eqref{Est_quad}, no individual $H^s$ bound was established. In section \ref{SecApp}, we will perform such bound transformation using the procedure that will be deployed.
\end{enumerate}
We obtain the following new estimates :
\begin{theorem}\label{Thm_App}
    We have that :
    \begin{itemize}
        \item For the $3D$ case, solutions of the cubic nonlinear Schr\"odinger equation obtained in \cite{KuksinShirikyan2004} satisfy the individual bound :
        \begin{align}
    \|u^\omega(t)\|_{H^2}\leq C^\omega (1+t)^{\frac{11+}{12}}\quad \forall t\geq 0.
\end{align}
        \item Solutions of the energy supercritical nonlinear Schr\"odinger equation obtained in \cite{sy2021almost} satisfy the individual bound :
       \begin{align}
    \|u^\omega(t)\|_{H^s}\leq C^\omega \left(\xi(1+\sqrt{\ln(1+t)})\right)^{p-1}(1+t)^{\frac{1+}{2}}\quad t\geq 0.
\end{align}
    \end{itemize}
\end{theorem}

Below, we present some tools, of independent interest, that form the basis of the procedure presented in Section \ref{SecApp} giving the proof of Theorem \ref{Thm_App}. Let $(\Omega,\,\mathcal{F},\, \mathbb{P})$ be a complete probability space, and let $(X,\, \|\cdot\|_X)$ be a Banach space. We have
\begin{theorem}\label{Thm:main}
    Let $u:\Omega\times\R_+\to X$ be a stochastic process valued in the Banach space $X$. Assume that
    \begin{enumerate}
        \item for $\mathbb{P}$-a. a. $\omega$, $u^\omega$ has the local increment
        \begin{align}
            \|u^\omega(t)-u^\omega(t')\|_X\leq A^\omega\phi(t) \quad \forall\, |t-t'|\leq 1,\label{Inc_control}
        \end{align}
        with $A^\omega<\infty$ and $\phi:\R_+\to\R_+$ an increasing function ;
        \item $u$ satisfies the following ensemble bound
        \begin{align}
            \mathbb{E}\psi(\|u(t)\|_X) \leq\lambda(t)\quad \forall t\geq 0,\label{Ens_bound}
        \end{align}
        for increasing functions $\psi,\, \lambda:\R_+\to\R_+$.
    \end{enumerate}
    Then, for $\mathbb{P}$-almost all $\omega$, there is $C^\omega<+\infty$ such that
        \begin{align}
            \|u^\omega(t)\|_X\leq C^\omega(\phi(t)+\psi^{-1}(\lambda(t)(1+t)^{1+}))\quad \forall t\geq 0.\label{Ind_bound}
        \end{align}
\end{theorem}
In the particular case of stationary processes, we have the following.
\begin{corollary}
    Let $u:\Omega\times\R_+\to X$ be a stationary stochastic process valued in the Banach space $X$. Assume that
    \begin{enumerate}
        \item for $\mathbb{P}$-a. a. $\omega$, $u^\omega$ has the local increment
        \begin{align}
            \|u^\omega(t)-u^\omega(t')\|_X\leq A^\omega\phi(t) \quad \forall |t-t'|\leq 1,
        \end{align}
        with $A^\omega<\infty$ and $\phi:\R_+\to\R_+$ an increasing function ;
        \item $u$ satisfies the following ensemble bound
        \begin{align}
            \mathbb{E}\psi(\|u(0)\|_X) <+\infty,
        \end{align}
        for an increasing function $\psi:\R_+\to\R_+$.
    \end{enumerate}
    Then, for $\mathbb{P}$-almost all $\omega$, there is $C^\omega<+\infty$ such that
        \begin{align}
            \|u^\omega(t)\|_X\leq C^\omega(\phi(t)+\psi^{-1}((1+t)^{1+}))\quad \forall t\geq 0.
        \end{align}
\end{corollary}
\begin{remark}
    From the increment condition \eqref{Inc_control}, we could use an integration argument to find the bound
    \begin{align}
        \|u^\omega(t)\|_X\lesssim_\omega \int_0^t\phi(\tau)\,d\tau ;
    \end{align}
    This bound would typically dominate $\phi(t)$ for $t>>1$. We can therefore see the role of the second condition \eqref{Ens_bound} in making it possible to avoid such integration. We will see this clearer in the applications where  the term $\psi^{-1}((1+t)^{1+})$ can be dropped as it will be smaller than $\phi(t)$ for $t>>1$.
\end{remark}

\begin{remark}
    The moral of the inequality \eqref{Ind_bound} is that in order to apply the reciprocal function in the inequality \eqref{Ens_bound} (as we could do if there was no expectation in the LHS), we have to pay the constraint imposed by the ensemble integral; the trade-off costs us $(1+t)^{1+}$, which is the guaranty of integrability in time. We are then paying in time an ensemble constraint. We also need to pay the local increment $\phi(t)$, which in our examples below will be the dominant part.
\end{remark}
\begin{remark}
    As we will notice in the next section, the proof of Theorem \ref{Thm:main} is quite simple and combines the Borel-Cantelli lemma,  with a crucial use of the local increment. While the proof is rather elementary, the consequences of these bounds are of great interest, as we will see in the estimation procedure presented in Section \ref{SecApp}. 
\end{remark}

\section{Proof of Theorem \ref{Thm:main}\label{SectProof}}
The following fact is a simple application of the Borel-Cantelli Lemma, and will be used several times in rest of the paper.
\begin{lemma}\label{LemBC}
Let $F:\Omega\times\mathbb{N}\to \R_+ $ be a measurable function and $f:\R_+\to (0,\infty)$ be an increasing function. Consider the sequence of sets
\begin{align}
   E_k=\{\omega\in\Omega,\, F(\omega,k)> f(k)\}.
\end{align}
Assume that
\begin{align}
    \sum_{k\in\N}\mathbb{P}(E_k) <\infty,\label{BC_cond}
\end{align}
then for $\mathbb{P}$-almost all $\omega$, there is $C^\omega>0$ such that
    \begin{align}
        F(\omega,k)\leq C^\omega f(k)\quad \forall k\in\N.
    \end{align}
\end{lemma}
\begin{proof}
    The argument is essentially an application of the Borel-Cantelli lemma. From the condition \eqref{BC_cond}, we infer that for $\mathbb{P}$-a. a. $\omega$ there is $k^\omega\in\N$ such that
    \begin{align}
        F(\omega,k)\leq f(k)\quad \forall k\geq k^\omega.
    \end{align}
    For this fixed $k^\omega$, we have that $\max_{k\leq k^\omega}F(\omega,k)=M^\omega<\infty.$ We arrive at the result.
    \end{proof}
\begin{proof}[Proof of Theorem \ref{Thm:main}]
 Let us consider the events $E_k=\{\omega\in\Omega,\, \|u^\omega(k)\|_X>\psi^{-1}(\lambda(k)(1+k)^{1+})\}$ for $k\in\mathbb{N}.$
    By Chebyshev inequality and the ensemble bound \eqref{Ens_bound}, we have that 
    \begin{align}
        \mathbb{P}(E_k)\leq \frac{\lambda(k)}{\lambda(k)(1+k)^{1+}}= \frac{1}{(1+k)^{1+}}.
    \end{align}
    It then follows from Lemma \ref{LemBC} that $\mathbb{P}$-almost surely, there is a constant $C_1^\omega >0$ such that
    \begin{align}
       \|u^\omega(k)\|_X\leq C^\omega_1\psi^{-1}(\lambda(k)(1+k)^{1+})\quad \forall k\in\mathbb{N}.
    \end{align}
     Now, let $t\in\R_+$. There is $k$ such that $k\leq t\leq k+1$. We have that
    \begin{align}
        \|u^\omega(t)\|_X\leq \|u^\omega(t)-u^\omega(k)\|_X+\|u^\omega(k)\|_X\leq \|u^\omega(t)-u^\omega(k)\|_X+C^\omega_1\psi^{-1}(\lambda(k)(1+k)^{1+}).
    \end{align}
    From the condition \eqref{Inc_control}, we obtain
    \begin{align}
        \|u^\omega(t)-u^\omega(k)\|_X\leq A^\omega\phi(t).
    \end{align}
    Now, since the functions $\psi^{-1}$ and $\lambda$ are increasing, we arrive at
    \begin{align}
        \|u^\omega(t)\|_X\leq A^\omega\phi(t)+ C^\omega_1\psi^{-1}((1+k)^{1+})\leq\tilde{C}^\omega(\phi(t)\psi^{-1}(\lambda(t)(1+t)^{1+})).
    \end{align}
\end{proof}

\section{Estimation procedure and Pathwise bounds\label{SecApp}}
In this section, we deploy an estimation procedure making use of Theorem \ref{Thm:main} to obtain individual bounds from ensemble estimates in the examples exposed in the introduction, proving the claims of Theorem \ref{Thm_App}.

\subsection{Situation 1 : $H^2$ estimate for probabilistic strong solutions of the cubic NLS [Kuksin and Shirikyan \cite{KuksinShirikyan2004}]} The authors of \cite{KuksinShirikyan2004} conidered the Schr\"odinger equation
\begin{align}
    \partial_t u=i(\Delta -|u|^2)u
\end{align}
in dimensions $d\leq 4$.
Here we only consider the case $d=3$ (the case $d<3$ is simpler while the case $d=4$ is out of reach of the estimate below). They constructed a probability space $(\Omega,\mathcal{F},\mathbb{P})$ and a stationary process $u:\Omega\times\R_+\to H^2$ such that $\mathbb{P}$-almost all $\omega\in\Omega$ $u^\omega\in L^2_{loc,t}H^2_x\cap W^{1,\frac{4}{3}}_{loc,t}L^{\frac{4}{3}}_x$ (\cite{KuksinShirikyan2004}, Theorem 2.4). This result does not rely on any local wellposedness, but rather on compactness argument employing Prokhorov and Skorokhod theorems. The process satisfies (inviscid limit combined with Proposition 1.5 in \cite{KuksinShirikyan2004})
\begin{align}
    \mathbb{E}\left(E^2(u(0))+\||u(0)|^2\|_{H^1}^2+\|u(0)\|_{H^2}^2\right) <\infty.
\end{align}
where $E(u)=\frac{1}{2}\|\nabla u\|_{L^2}^2+\frac{1}{4}\|u\|_{L^4}^4$ is the energy of the equation, which is proven to be conserved in time.
By Agmon's inequality, we have
\begin{align}
    \mathbb{E}\|u(0)\|_{L^\infty}^4<\infty. \label{NLS_Ens_Linfty}
\end{align}
We have for any $\omega$ in some full probability set $\Omega_*$, that $u^\omega(0)\in H^2$ and $u^\omega(\cdot)\in C_tH_x^{-\frac{3}{4}}\cap L^2_{loc,t}H^2_x\cap L^4_{loc,t}L^6_x$. We  write the  Duhamel formulation 
\begin{align}
    u^\omega(t)-u^\omega(t')=e^{it\Delta}u^\omega(0)-e^{it'\Delta}u^\omega(0)-i\int_{t'}^te^{i(t-\tau)\Delta}|u^\omega(\tau)|^2u^\omega(\tau)\,d\tau.
\end{align}
We want to estimate the increments in the $L^\infty$-norm.
Recall that the Schr\"odinger group $e^{it\Delta}$ is not $L^\infty\to L^\infty$-continuous. However, we combine Agmon's inequality and the product estimate to find
\begin{align}
    \|e^{it\Delta}|u|^2u\|_{L^\infty}\leq C\|e^{it\Delta}|u|^2u\|_{H^1}^\frac{1}{2}\|e^{it\Delta}|u|^2u\|_{H^2}^\frac{1}{2}=C\|u|^2u\|_{H^1}^\frac{1}{2}\||u|^2u\|_{H^2}^\frac{1}{2}&\leq C_1\|u\|_{L^\infty}^2\|u\|_{H^1}^\frac{1}{2}\|u\|_{H^2}^\frac{1}{2}\\
    &\leq C_2\|u\|_{H^1}^\frac{3}{2}\|u\|_{H^2}^\frac{3}{2}.
\end{align}
We then have, for $t'\leq t,$
\begin{align}
    \|u^\omega(t)-u^\omega(t')\|_{L^\infty} &\leq \|e^{it\Delta}u^\omega(0)-e^{it'\Delta}u^\omega(0)\|_{H^2}+\int_{t'}^t\|e^{i(t-\tau)\Delta}|u^\omega(\tau)|^2u^\omega(\tau)\|_{L^\infty}\,d\tau\\
    &\leq 2\|u^\omega(0)\|_{H^2}+CE(u^\omega(0))^\frac{3}{4}\int_{t'}^t\|u^\omega(\tau)\|_{H^2}^{\frac{3}{2}}\,d\tau.\label{Linfty_interm}
\end{align}
\begin{lemma}
    We have that : for $\mathbb{P}$-a. a. $\omega$, there is $C^\omega<\infty$ such that
    \begin{align}
        \int_{t'}^t\|u^\omega(\tau)\|_{H^2}^{\frac{3}{2}}\,d\tau\leq C^\omega(1+t)^{\frac{3+}{4}}\quad  \forall\, |t-t'|\leq 1.\label{Inc_NLS}
    \end{align}
\end{lemma}
\begin{proof}
    Define the sets 
    \begin{align}
        E_k=\left\{\omega\in\Omega,\, \int_k^{k+1}\|u^\omega(\tau)\|_{H^2}^\frac{3}{2}>(1+k)^\frac{3+}{4}\right\}.
    \end{align}
    By the Chebyshev and Jensen inequalities, we have that
    \begin{align}
        \mathbb{P}(E_k)\leq \frac{\mathbb{E}(\int_k^{k+1}\|u^\omega(\tau)\|_{H^2}^\frac{3}{2}\,d\tau)^{\frac{4}{3}}}{(1+k)^{1+}}\leq \frac{C}{(1+k)^{1+}}.
    \end{align}
    Applying Lemma \ref{LemBC}, we obtain that : for $\mathbb{P}$-a. a. $\omega$, there is $C^\omega_1$
    \begin{align}
        \int_k^{k+1}\|u^\omega(\tau)\|_{H^2}^\frac{3}{2}\,d\tau \leq C^\omega_1(1+k)^{\frac{3+}{4}}\quad \forall k\in\N.
    \end{align}
    Let $t,\, t'\in \R_+$ be such that $|t-t'|\leq 1$, then there is $k\in\N$ such that $k\leq t,\, t'\leq k+2$. Therefore, we obtain
    \begin{align}
        \int_{t'}^t\|u^\omega(\tau)\|_{H^2}^\frac{3}{2}\,d\tau \leq \int_k^{k+2}\|u^\omega(\tau)\|_{H^2}^\frac{3}{2}\,d\tau \leq C^\omega(1+t)^\frac{3+}{4}.
    \end{align}
    In the last inequality we use the fact that $t\geq k$.
\end{proof}
We resume from \eqref{Linfty_interm} to arrive at
\begin{align}
 \|u(t)-u(t')\|_{L^\infty} &\leq A^\omega(1+t)^{\frac{3+}{4}}\quad \forall\, |t-t'|\leq 1.    
\end{align}
Combining this with the ensemble bound for the $L^\infty$-norm \eqref{NLS_Ens_Linfty}, we satisfy the conditions of Theorem \ref{Thm:main}; we obtain 
\begin{align}
    \|u^\omega(t)\|_{L^\infty} \leq C_1^\omega((1+t)^\frac{3+}{4}+(1+t)^\frac{1+}{4})\leq C^\omega (1+t)^\frac{3+}{4}.
\end{align}
 To obtain a pathwise $H^2$ bound, since we already have the $H^2$ ensemble bound, it remains to find a corresponding local increment increment : let us write, for $t,\, t'\in\R_+$ satisfying $0\leq t-t'\leq 1$,
 \begin{align}
     \|u^\omega(t)-u^\omega(t')\|_{H^2} &\leq 2\|u_0^\omega\|_{H^2}+C\int_{t'}^t\|u^\omega(\tau)\|_{L^\infty}^2\|u^\omega(\tau)\|_{H^2}\,d\tau\\
     &\leq 2\|u_0^\omega\|_{H^2}+C^\omega_1(1+t)^{\frac{3+}{4}}\int_{t'}^t\|u^\omega(\tau)\|_{H^2}\,d\tau.
 \end{align}
Employing Lemma \ref{LemBC} exactly as in the proof of \eqref{Inc_NLS}, we find that for $\mathbb{P}$-a. a. $\omega$, there is $C_1^\omega >0$ such that
\begin{align}
    \int_{t'}^t\|u^\omega(\tau)\|_{H^2}\,d\tau \leq C_1^\omega(1+t)^\frac{1+}{2} \quad |t-t'|\leq 1.
\end{align}
Overall, we can redefine a new set of probability $1$, on which
\begin{align}
    \|u^\omega(t)-u^\omega(t')\|_{H^2} &\leq C_2^\omega (1+t)^{\frac{11+}{12}}.
\end{align}
Combining this increment estimate with the bound $\mathbb{E}\|u\|_{H^2}^2<\infty$ in the spirit of Theorem \ref{Thm:main}, we obtain the $\mathbb{P}$-almost sure bound
\begin{align}
    \|u^\omega(t)\|_{H^2}\leq C^\omega (1+t)^{\frac{11+}{12}}\quad \forall t\geq 0.
\end{align}

\subsection{ Situation 2 : $H^s$ Estimate for supercritical NLS \cite{sy2021almost}]} In \cite{sy2021almost}, a probabilistic global wellposedness for energy supercritical NLS was established on $H^s(\mathbb{T}^3)$ for $s>\frac{3}{2}$ :
\begin{align}
    \partial_t u=i(\Delta -|u|^{p-1})u.
\end{align}
In particulat, the following estimates were proved
\begin{align}
    \mathbb{E}\left(\|u\|_{H^s}^2\right) <\infty,
\end{align}
\begin{align}
    \|u^\omega(t)\|_{H^{s-}}\leq C^\omega \xi(1+\sqrt{\ln(1+t)})\label{Control_fromMS}.
\end{align}
This result does rely on a wellposedness theory in a procedure called IID limit, and estimate \eqref{Control_fromMS} exploits the Bourgain procedure. However, no individual estimate on the limiting regularity, $H^s$, was obtained because of lack of a strong enough ensemble bound at this regularity (we only have an ensemble bound on the quadratic of the $H^s$-norm which is insufficient for the Bourgain framework.) Here we establish such individual estimate by employing the described procedure. We start with the increment property, for $t,\, t'\in\R_+$ satisfying $|t-t'|\leq 1$. Recall that $s>\frac{3}{2}$.
\begin{align}
    \|u^\omega(t)-u^\omega(t')\|_{H^s}\leq 2\|u^\omega(0)\|_{H^s}+\int_{t'}^t\|u^\omega(\tau)\|_{\frac{3}{2}+}^{p-1}\|u^\omega(\tau)\|_{H^s}\,d\tau
\end{align}
We use the bound \eqref{Control_fromMS}, we obtain
\begin{align}
    \|u^\omega(t)-u^\omega(t')\|_{H^s}\leq 2\|u^\omega(0)\|_{H^s}+ A^\omega 
    (\xi(1+\ln(1+t)))^{p-1}\int_{t'}^t\|u^\omega(\tau)\|_{H^s}\,d\tau.
\end{align}
As in \eqref{Inc_NLS}, we have that
\begin{align}
    \int_{t'}^t\|u^\omega(\tau)\|_{H^s}\,d\tau \leq C_2^\omega(1+t)^{\frac{1+}{2}}\quad |t-t'|\leq 1.
\end{align}
We obtain the increment : 
\begin{align}
    \|u^\omega(t)-u^\omega(t')\|_{H^s}\leq B^\omega (\xi(1+\ln(1+t)))^{p-1}(1+t)^{\frac{1+}{2}}\quad |t-t'|\leq 1.
\end{align}
Combining the above inequality with the ensemble bound $\mathbb{E}\|u(0)\|_{H^s}^2<\infty$, we apply Theorem \ref{Thm:main} to arrive at
\begin{align}
    \|u^\omega(t)\|_{H^s}\leq C^\omega \left(\xi(1+\sqrt{\ln(1+t)})\right)^{p-1}(1+t)^{\frac{1+}{2}}\quad t\geq 0.
\end{align}
\section*{Acknowledgments}
The research of M. Sy is funded by the Alexander von Humboldt foundation under the “German Research Chair programme” financed by the Federal Ministry of Education and Research (BMBF). 
\bibliography{ref}
\bibliographystyle{abbrv}
\end{document}